\documentclass[reqno, 12pt, reqno]{amsart}

\usepackage{amsfonts, amsthm, amsmath, amssymb,amsrefs}
\usepackage{hyperref}
\hypersetup{colorlinks=false}

\usepackage[margin=1.5in]{geometry}
\usepackage{graphicx,rotating}
\usepackage{wrapfig}

\RequirePackage{mathrsfs} \let\mathcal\mathscr

\numberwithin{equation}{section}

\newtheorem{theorem}{Theorem}[section]
\newtheorem{lemma}[theorem]{Lemma}
\newtheorem{proposition}[theorem]{Proposition}

\theoremstyle{definition}

\newtheorem{remark}[theorem]{Remark}

\newcommand{\R}{\mathbb{R}}

\renewcommand{\le}{\leqslant}

\renewcommand{\ge}{\geqslant}

\renewcommand{\Re}{\operatorname{Re}}
\renewcommand{\Im}{\operatorname{Im}}

\DeclareMathOperator{\sn}{sn}

\begin{document}

\date{\today}

\title{True value of an integral in Gradshteyn and Ryzhik's table.}
\author{Juan\ Arias de Reyna}

\address{Univ.~de Sevilla \\
Facultad de Matem\'aticas \\
c/Tarfia, sn
 \\
41012-Sevilla \\
Spain}
\email{arias@us.es}


\begin{abstract}
Victor Moll pointed out that 
entry 3.248.5 in the sixth edition of Gradshteyn and Ryzhik tables
of integrals was incorrect. He asked some years ago what was the true value of this integral.
I evaluate it in terms of two elliptic  integrals. The evaluation
is standard but involved, using real and complex analysis.
\end{abstract}

\maketitle
\thispagestyle{empty}

\setcounter{tocdepth}{1}
\tableofcontents

\section{Introduction.}

Victor Moll in several places (\cite{BM}, \cite{Borwein}*{p.~184}, \cite{AmdMoll}, \cite{moll2}, \cite{mollB1}*{p.~176--177}) tell the story of how he started the project of proving each formula in Gradshteyn and Ryzhik table of integrals. For example we quote from  \cite{mollB1}

%\begin{quote}
%A literature search shows that the table of integrals by I. S. Gradshteyn and I. M. Ryzhik is one of the most used by the scientific community. The author became interested in the verification of its entries while trying to verify entry $\mathbf{3.248.5}$ of \cite{gr1}
%\begin{equation}\tag{26}\label{3.248.5}
%\int_0^\infty\frac{dx}{(1+x^2)^{3/2}\bigl[\varphi(x)+\sqrt{\varphi(x)}\bigr]^{1/2}}
%=\frac{\pi}{2\sqrt{6}},
%\end{equation}
%with
%\[\varphi(x)=1+\frac{4x^2}{3(1+x^2)^2}.\]
%The table gave the value $\pi/2\sqrt{6}$. A direct numerical integration shows that this is incorrect. The integral is approximately 0.666377 and the right hand side is about 0.641275. This error produces two natural questions: what is the value of the integral? and what produced the incorrect answer? To this day, April 2014, the author does not know how to answer either one.
%\end{quote}

\begin{quote}
Given the large number of entries in \cite{gr1}, we have not yet developed an order in which to check them. Once in a while an entry catches our eye. This was the case with entry 
$\mathbf{3.248.5}$ in the sixth edition of the table by Gradshteyn and Ryzhik. The presence of the double square root in the appealing integral
\begin{equation}\tag{26}\label{3.248.5}
\int_0^\infty\frac{dx}{(1+x^2)^{3/2}\bigl[\varphi(x)+\sqrt{\varphi(x)}\bigr]^{1/2}}
=\frac{\pi}{2\sqrt{6}},
\end{equation}
with
\[\varphi(x)=1+\frac{4x^2}{3(1+x^2)^2},\]
remind us of (16). \emph{Unfortunately \eqref{3.248.5} is incorrect}. The numerical value of the left-hand side is approximately $0.666377$, and the right-hand side is about $0.641275$. The table \cite{gr1} is continually being revised. After we informed the editors of the error in $\mathbf{3.248.5}$, it was taken out. There is no entry $\mathbf{3.248.5}$ in \cite{gr2}. At the present time, we are still reconciling this formula.
\end{quote}

There are two natural questions about entry $\mathbf{3.248.5}$. What is the true value 
of this integral? and, there is a variation of the integrand that integrates to $\pi/2\sqrt{6}$?
Our purpose here is to compute an exact value for this integral in terms of elliptic integrals at well defined arguments. We still have no good answer for the second question.

In all the paper we call $I$ the value of the integral. Our corrected entry is
\[\int_0^{\infty}\frac{dx}{(1+x^2)^{3/2}[\varphi(x)+\sqrt{\varphi(x)}]^{1/2}}=
\frac{\sqrt{3}-1}{\sqrt{2}}\Pi(\pi/2,k,3^{-1/2})-\frac{1}{\sqrt{2}}F(\alpha, 3^{-1/2}),
\]
where $k=2-\sqrt{3}$,  $\alpha=\arcsin\sqrt{k}$ and 
$F(\varphi,k)$ and  $\Pi(\varphi,n,k)$ are the elliptic integral of the first and 
third kind respectively. 

We have used here the notations in Gradshteyn and Ryzhik tables, so that
\[F(\alpha, 3^{-1/2})=\int_0^{\sqrt{k}}\frac{dx}{\Delta(x)},\qquad 
\Pi(\pi/2,k,3^{-1/2})=\int_0^1\frac{dx}{(1-kx^2)\Delta(x)},\]
where  
$\Delta(x)=\sqrt{(1-x^2)(1-\frac{1}{3}x^2)}$.

\section{Representation as a conditionally convergent double series.}

Recall that  $I$ is  the value of the integral. First we need to transform it a little
and then we obtain the series representation.

\begin{proposition}
We have 
\begin{equation}\label{E:1}
I=\int_0^1 \frac{dy}{
\left[1+\frac43(y^2-y^4)+\left(1+
\frac43(y^2-y^4)\right)^{1/2}\right]^{1/2}}.
\end{equation}
\end{proposition}
\begin{proof}
We have $\varphi(1/y)=\varphi(y)$, so that the change of variables 
$y=x^{-1}$ leaves the integral almost equal
\[I=\int_0^\infty \frac{y\,dy}{(1+y^2)^{3/2}\left[1+\frac{4y^2}{3(1+y^2)^2}+\left(
1+\frac{4y^2}{3(1+y^2)^2}\right)^{1/2}\right]^{1/2}}.\]
After this the change  $x=1+y^2$ yields
\[I=\frac12\int_1^\infty \frac{dx}{x^{3/2}
\left[1+\frac{4x-4}{3x^2}+\left(
1+\frac{4x-4}{3x^2}\right)^{1/2}\right]^{1/2}}.\]
Finally we change again by means of $xy^2=1$ to get \eqref{E:1}.
\end{proof}

We have
\begin{proposition}
\begin{equation}\label{E:2}
I=\sum_{n=0}^\infty(-1)^n\frac{1}{2^{2n}}\binom{2n}{n}\sum_{k=0}^\infty 
\frac{(-1)^k}{k!}\frac{\Gamma(\frac{n+1}{2}+k)}{\Gamma(\frac{n+1}{2})}2^{2k}
\frac{2^{2k}}{3^k}\frac{(2k)!(2k)!}{(4k+1)!}.
\end{equation}
The series is not absolutely convergent, so the order of the sums is important here.
\end{proposition}
\begin{proof}
Let $h(y)=1+\frac43(y^2-y^4)$, \eqref{E:1} can be written as 
\[I=\int_0^1\frac{dy}{h(y)^{1/2}(1+h(y)^{-1/2})^{1/2}}.\]
For $0<y<1$ we have $h(y)>1$, therefore we may expand $(1+h(y)^{-1/2})^{-1/2}$ in power
series
\[I=\int_0^1\frac{dy}{h(y)^{1/2}}\sum_{n=0}^\infty\binom{-1/2}{n}h(y)^{-n/2}.\]
The binomial series $\sum_{n=0}^\infty\binom{-1/2}{n}x^n$ converges at $x=1$, therefore
this series converges uniformly on  $[0,1]$. For $y\in[0,1]$ we have $1\le h(y)\le 4/3$,
so that the above series converges uniformly and can be integrated term by term.
\[I=\sum_{n=0}^\infty\binom{-1/2}{n}\int_0^1 h(y)^{-(n+1)/2}\,dy.\]
For $y\in[0,1]$ we have $0\le\frac43(y^2-y^4)\le \frac13$, so that again we can expand
$h(y)^{-(n+1)/2}$ applying Newton binomial series
\[I=\sum_{n=0}^\infty\binom{-1/2}{n}\int_0^1 \sum_{k=0}^\infty\binom{-(n+1)/2}{k}
\frac{4^k}{3^k}y^{2k}(1-y^2)^{k}\,dy.\]
Again this series converges uniformly and can be integrated term by term
\[I=\sum_{n=0}^\infty\binom{-1/2}{n} \sum_{k=0}^\infty\binom{-(n+1)/2}{k}
\frac{4^k}{3^k}\int_0^1y^{2k}(1-y^2)^{k}\,dy.\]
The integral can be reduced to a binomial (\cite{gr2}*{3.251.1}) so that
\[\int_0^1y^{2k}(1-y^2)^{k}\,dy=\frac12B\Bigl(k+\frac12,k+1\Bigr)=
2^{2k}\frac{(2k)!\,(2k)!}{(4k+1)!}.\]
Applying also that
\[\binom{-1/2}{n}=(-1)^n\frac{1}{2^{2n}}\binom{2n}{n},\quad
\binom{-(n+1)/2}{k}=\frac{(-1)^k}{k!}\frac{\Gamma(\frac{n+1}{2}+k)}{\Gamma(\frac{n+1}{2})},
\]
yields
\[I=\sum_{n=0}^\infty (-1)^n\frac{1}{2^{2n}}\binom{2n}{n}
\sum_{k=0}^\infty \frac{(-1)^k}{k!}\frac{\Gamma(\frac{n+1}{2}+k)}{\Gamma(\frac{n+1}{2})}
\frac{4^k}{3^k}2^{2k}\frac{(2k)!\,(2k)!}{(4k+1)!}.\]
\end{proof}

\section{Transforming the series into  a triple integral.}

We will need a Lemma in the proof of Proposition \ref{P:newintegral}.
\begin{lemma}\label{L:1}
For any $t>0$ we have 
\begin{equation}\label{E:4}
U(t):=\sum_{k=0}^\infty 
\frac{(-1)^k}{k!}2^{2k}\frac{(2k)!(2k)!}{(4k+1)!}\frac{4^kt^k}{3^k}=
\int_0^1 e^{-\frac{16}{3}u^2(1-u)^2t}\,du.
\end{equation}
There is a constant $C$ such that for $t>1$ we have $0< U(t)\le C t^{-1/2}$.
\end{lemma}
\begin{proof}
Notice that 
\[\frac{(2k)!(2k)!}{(4k+1)!}=\int_0^1(u(1-u))^{2k}\,du,\]
so that 
\[\sum_{k=0}^\infty 
\frac{(-1)^k}{k!}2^{2k}\frac{(2k)!(2k)!}{(4k+1)!}\frac{4^kt^k}{3^k}=
\sum_{k=0}^\infty 
\frac{(-1)^k}{k!}2^{2k}\int_0^1(u(1-u))^{2k}\,du\frac{4^kt^k}{3^k}.\]
It is easy to justify that we may here interchange the order of sum and integral so that 
\[\sum_{k=0}^\infty 
\frac{(-1)^k}{k!}2^{2k}\frac{(2k)!(2k)!}{(4k+1)!}\frac{4^kt^k}{3^k}=\int_0^1
e^{-\frac{16}{3}u^2(1-u)^2t}\,du.\]
For $0<u<1/2$ we have $(1-u)^2>1/4$, so that 
\[U(t)\le \int_0^{1/2}e^{-\frac{4}{3}u^2t}\,du\le \int_0^{\infty}e^{-4u^2t/3}\,du=\frac{\sqrt{3\pi}}{4\sqrt{t}}.\]
\end{proof}

\begin{proposition}\label{P:newintegral}
We have
\begin{equation}\label{E:3}
I=\int_0^\infty\Bigl(\sum_{n=0}^\infty\frac{(-1)^n}{2^{2n}}\binom{2n}{n}
\frac{t^{\frac{n-1}{2}}}{\Gamma(\frac{n+1}{2})}\Bigr)\Bigl(\sum_{k=0}^\infty 
\frac{(-1)^k}{k!}2^{2k}\frac{(2k)!(2k)!}{(4k+1)!}\frac{4^kt^k}{3^k}\Bigr)e^{-t}\,dt.
\end{equation}
\end{proposition}
\begin{proof}
Let $U(t)$ be the series considered in Lemma \ref{L:1}, by this Lemma we have 
$0\le U(t)\le 1$ for any $t>0$.  We will apply the dominated convergence theorem to prove 
that 
\begin{multline*}
\lim_{N\to\infty}
\int_0^\infty\Bigl(\sum_{n=0}^N\frac{(-1)^n}{2^{2n}}\binom{2n}{n}
\frac{t^{\frac{n-1}{2}}}{\Gamma(\frac{n+1}{2})}\Bigr)U(t)e^{-t}\,dt\\=
\int_0^\infty\Bigl(\sum_{n=0}^\infty\frac{(-1)^n}{2^{2n}}\binom{2n}{n}
\frac{t^{\frac{n-1}{2}}}{\Gamma(\frac{n+1}{2})}\Bigr)U(t)e^{-t}\,dt
\end{multline*}
Assuming this,  the right hand side of equation \eqref{E:3} is equal to 
\begin{equation}\label{E:step1}
\sum_{n=0}^\infty\frac{(-1)^n}{2^{2n}}\binom{2n}{n}
\frac{1}{\Gamma(\frac{n+1}{2})}\int_0^\infty t^{\frac{n-1}{2}}U(t)e^{-t}\,dt
=\sum_{n=0}^\infty\frac{(-1)^n}{2^{2n}}\binom{2n}{n}
\frac{a_n}{\Gamma(\frac{n+1}{2})},
\end{equation}
where 
\[a_n=\int_0^\infty \sum_{k=0}^\infty 
\frac{(-1)^k}{k!}2^{2k}\frac{(2k)!(2k)!}{(4k+1)!}\frac{4^kt^{\frac{n-1}{2}+k}}{3^k}
e^{-t}\,dt.\]
We may integrate term by term, since in this case the integral of the absolute value 
of the terms of the series have a finite sum.
\[a_n=\sum_{k=0}^\infty 
\frac{(-1)^k}{k!}2^{2k}\frac{(2k)!(2k)!}{(4k+1)!}\frac{4^k \Gamma(\frac{n+1}{2}+k)}{3^k}.\]
Inserting this value in \eqref{E:step1} we get that the right hand side of \eqref{E:3}
is equal to the right hand side of \eqref{E:2}, and therefore to our integral $I$. 

It remains to justify the application of  the dominated convergence theorem above. 
Consider the series $\sum_{n=1}^\infty (-1)^n f(n)$ with 
\[f(n):=\frac{1}{2^{2n}}\binom{2n}{n}\frac{t^{\frac{n-1}{2}}}{\Gamma(\frac{n+1}{2})}.\]
The sequence $f(n)$ for $n\ge1$ is increasing for $n\le2t-8$ and decreasing for $n\ge2t$. To see 
it we use Gautschi's inequalities \cite{DLMF}*{\href{http://dlmf.nist.gov/5.6.E4}{5.6.4}}.
For $n\le2t-8$ and $n\ge1$ we have
\begin{multline*}
\frac{f(n)}{f(n+1)}=\frac{2n+2}{2n+1}\frac{\Gamma(\frac{n}{2}+1)}
{\Gamma(\frac{n}{2}+\frac12)}\frac{1}{\sqrt{t}}\le\Bigl(\frac{2n+2}{2n+1}\Bigr)\Bigl(1+\frac{n}{2}\Bigr)^{\frac12}\frac{1}{\sqrt{t}}\\\le 
\Bigl(\frac{n}{2t}\Bigr)^{1/2}\Bigl(\frac{2n+2}{2n+1}\Bigr)\Bigl(1+\frac{2}{n}\Bigr)^{1/2}
\le \Bigl(\frac{n}{2t}\Bigr)^{1/2}\Bigl(1+\frac{2}{n}\Bigr)\le 1.
\end{multline*}
The last inequality is true because
\[\frac{n}{2t}\Bigl(1+\frac2n\Bigr)^2=\frac{n}{2t}+\frac{2}{t}+\frac{2}{nt}\le 1-\frac{4}{t}+\frac{2}{t}+\frac{2}{t}\le1\]
For $n\ge 2t$ we have
\[
\frac{f(n)}{f(n+1)}=\frac{2n+2}{2n+1}\frac{\Gamma(\frac{n}{2}+1)}
{\Gamma(\frac{n}{2}+\frac12)}\frac{1}{\sqrt{t}}\ge
\Bigl(\frac{2n+2}{2n+1}\Bigr)\Bigl(\frac{n}{2}\Bigr)^{\frac12}\frac{1}{\sqrt{t}}
\ge1\]
For a monotone sequence $|\sum_{n=N}^M (-1)^n a_n|\le \max_{N\le n\le M}|a_n|$. 
It follows that 
\[\sup_M\Bigl|\sum_{n=0}^M(-1)^n f(n)\Bigr|\le 11\sup_{n\ge0} f(n)\]

By Stirling formula we have 
\[\frac{1}{2^{2n}}\binom{2n}{n}\frac{t^{\frac{n-1}{2}}}{\Gamma(\frac{n+1}{2})}\sim
\frac{t^{\frac{n-1}{2}}}{\sqrt{2\pi}\; \Gamma(\frac{n}{2}+1)}\sim
\frac{t^{\frac{n-1}{2}}}{\sqrt{2\pi}}\frac{1}{\sqrt{\pi n}}\Bigl(\frac{2e}{n}\Bigr)^{n/2}.\]
Since the two functions are of the same order, the maximum of one is bounded by a constant
for the maximum of the other. 
For $t\gg1$ the maximum is attained approximately at $n=2t$ and is equal to
\[\frac{e^t}{2\pi t}\]
For $t<\frac12$ the sequence $f(n)$ is decreasing and the maximum is $f(0)>f(1)$. Therefore
for $t<1/2$ the maximum is $f(0)=\frac{1}{\sqrt{\pi t}}$. 
Applying that $0<U(t)<1$ for $0<t<1$ and $\le ct^{-1/2}$ for $t>1$ we get
\[\Bigl|
\Bigl(\sum_{n=0}^N\frac{(-1)^n}{2^{2n}}\binom{2n}{n}
\frac{t^{\frac{n-1}{2}}}{\Gamma(\frac{n+1}{2})}\Bigr)U(t)e^{-t}\Bigr|\le 
\begin{cases} 
\frac{c}{\sqrt{t}} & \text{ for $0<t<1$},\\
\frac{c}{t^{3/2}}& \text{ for $t>1$},
\end{cases}\]
for some constant $c$.  Therefore we have a uniform bound by an integrable function.
\end{proof}

Let $\Omega$ be the region equal to the complex plane with a cut along the negative real axis.
Let us denote by $\sqrt{z}$ the analytic function on $\Omega$ defined as the principal value, taking $|\arg(z)|<\pi$.
On $\Omega$ the function $z+\sqrt{z}$ is never negative, for example when $\Im(z)>0$ 
we have also $\Im(\sqrt{z})>0$.   Therefore $\sqrt{z+\sqrt{z}}$ is a well defined and analytic function on $\Omega$. We will use this notation everywhere in the paper.

\begin{proposition}
Let $\delta>0$,  for any $t>0$ we have
\begin{equation}\label{E:5}
\frac{1}{2\pi i}\int_H\frac{e^{tz}\,dz}{\sqrt{z+\sqrt{z}}}=\sum_{n=0}^\infty\frac{(-1)^n}{2^{2n}}\binom{2n}{n}
\frac{t^{\frac{n-1}{2}}}{\Gamma(\frac{n+1}{2})},
\end{equation}
where $H$ is Hankel's contour:  the boundary of the region containing the 
points with a distance $<\delta$ to the
negative real axis.
\end{proposition}

\begin{proof}
By Cauchy's Theorem the integral do not depend on $\delta>0$. Take $\delta=2$, for example,
then for $z$ in the path $H$ we have $\sqrt{z+\sqrt{z}}=\sqrt{z}\sqrt{1+1/\sqrt{z}}$, where all 
the square roots are principal values. Since $|1/\sqrt{z}|\le 2^{-1/2}$ we have
\[\frac{1}{\sqrt{z+\sqrt{z}}}=\sum_{n=0}^\infty \binom{-1/2}{n}(1/\sqrt{z})^{n+1}
=\sum_{n=0}^\infty \binom{-1/2}{n} z^{-\frac{n+1}{2}},\]
where $z^s:=e^{s\log z}$ with  $\log z$ meaning the principal logarithm in $\Omega$.

The series converges uniformly for $z$ in $H$, therefore we may integrate term by term and 
\[\frac{1}{2\pi i}\int_H\frac{e^{tz}\,dz}{\sqrt{z+\sqrt{z}}}=\sum_{n=0}^\infty \binom{-1/2}{n}\frac{1}{2\pi i}\int_H 
z^{-\frac{n+1}{2}}e^{tz}\,dz.\]
Since we assume $t>0$ $w=tz$ run through another Hankel contour when $z$ run through $H$
and the integral do not depend on which Hankel's contour we integrate, therefore
\[\frac{1}{2\pi i}\int_H\frac{e^{tz}\,dz}{\sqrt{z+\sqrt{z}}}=
\sum_{n=0}^\infty \binom{-1/2}{n} t^{\frac{n-1}{2}}\frac{1}{2\pi i}\int_H 
w^{-\frac{n+1}{2}}e^{w}\,dw.\]
Recalling the integral representation of the $\Gamma$ function
\href{https://dlmf.nist.gov/5.9.E2}{\cite{DLMF}*{\textbf{5.9.2}}}
\[\frac{1}{\Gamma(s)}=\frac{1}{2\pi i}\int_H z^{-s}e^z\,dz\]
we get \eqref{E:5}.
\end{proof}

For the next Proposition we pick a particular parametrization of $H$. We fix the value 
of $\delta=\frac12$, and fix a concrete  parametrization of  $H$, namely $\gamma(\xi)$ defined as follows.  
Take $\gamma(\xi)=\frac12(\xi+1-i)$ for $\xi\le -1$, $\gamma(\xi)=\frac12e^{\pi i \xi/2}$ for $-1<\xi<1$ and 
$\gamma(\xi)=\frac12(1-\xi+i)$ for $\xi>1$. Then any integral
\[\int_Hf(z)\,dz=\int_{-\infty}^\infty f(\gamma(\xi))\gamma'(\xi)\,d\xi\]
is an ordinary Lebesgue  integral in $\R$. Notice also that $|\gamma'(\xi)|\le \pi/4$ for all $\xi\in\R$.

\begin{proposition}
We have 
\begin{equation}\label{E:6}
I= \frac{1}{2\pi i}\int_0^\infty\int_H\int_0^1
 \frac{e^{tz}}{\sqrt{z+\sqrt{z}}} e^{-\frac{16}{3}u^2(1-u)^2t}e^{-t}\,du\,dz\,dt
\end{equation}
where the integral is an absolutely convergent triple integral. 
\end{proposition}
\begin{proof}
Substituting \eqref{E:4} and \eqref{E:5} into \eqref{E:3} we have
\[I= \int_0^\infty\Bigl(\frac{1}{2\pi i}\int_H \frac{e^{tz}\,dz}{\sqrt{z+\sqrt{z}}}\Bigr)\Bigl(\int_0^1 e^{-\frac{16}{3}u^2(1-u)^2t}\,du\Bigr)e^{-t}\,dt\]
Therefore we only need to show that the integral is absolutely convergent. We divide the 
integral in $\xi$ in three intervals. For $|\xi|\le1$ we have
\[\Bigl|\frac{e^{tz}}{\sqrt{z+\sqrt{z}}} e^{-\frac{16}{3}u^2(1-u)^2t}e^{-t}\gamma'(\xi)\Bigr|
\ll e^{\frac{t}{2}\cos\pi\xi/2}e^{-t}\le e^{-t/2};\]
and $e^{-t/2}$ have a finite integral  for $(\xi,u,t)\in[-1,1]\times[0,1]\times[0,\infty)$. 
For $\xi>1$ we have 
\[\Bigl|\frac{e^{tz}}{\sqrt{z+\sqrt{z}}} e^{-\frac{16}{3}u^2(1-u)^2t}e^{-t}\gamma'(\xi)\Bigr|
\ll \xi^{-1/2} e^{\frac{t}{2}(1-\xi)}e^{-t}\le \xi^{-1/2}e^{-t\xi/2-t/2};\]
which have a finite integral  on the set of $(\xi,u,t)\in[1,+\infty)\times[0,1]\times[0,\infty)$.
Notice that 
\[\int_1^\infty\xi^{-1/2}e^{-t\xi/2}e^{-t/2}\,d\xi\le \int_0^\infty\xi^{-1/2}e^{-t\xi/2}e^{-t/2}\,d\xi=
e^{-t/2}\sqrt{2\pi/t}.\]
Which is integrable for $t\in(0,+\infty)$.  The case $\xi<-1$ is treated in the same way.
\end{proof}

\section{Another simple integral.}

\begin{proposition}
We have 
\begin{equation}
I=\int_0^1\frac{dx}{\sqrt{1+\frac{16}{3}x^2(1-x)^2+\sqrt{1+\frac{16}{3}x^2(1-x)^2}}}.
\end{equation}
\end{proposition}

\begin{proof}
Applying Fubini's Theorem in \eqref{E:6} we may integrate first with respect to $t$, 
and we obtain
\[I=\int_0^1\Bigl(\frac{1}{2\pi i}\int_H \frac{dz}{\sqrt{z+\sqrt{z}}}\frac{1}{(1-z+\frac{16}{3}u^2(1-u)^2)}\Bigr)\,du.\]
Applying Cauchy's residue theorem the integral 
\[\int_{\Gamma_R}\frac{dz}{\sqrt{z+\sqrt{z}}(1-z+\frac{16}{3}u^2(1-u)^2)}=\int_C\frac{dz}{\sqrt{z+\sqrt{z}}(1-z+\frac{16}{3}u^2(1-u)^2)}.\]

\begin{figure}
\includegraphics[width=0.5\textwidth]{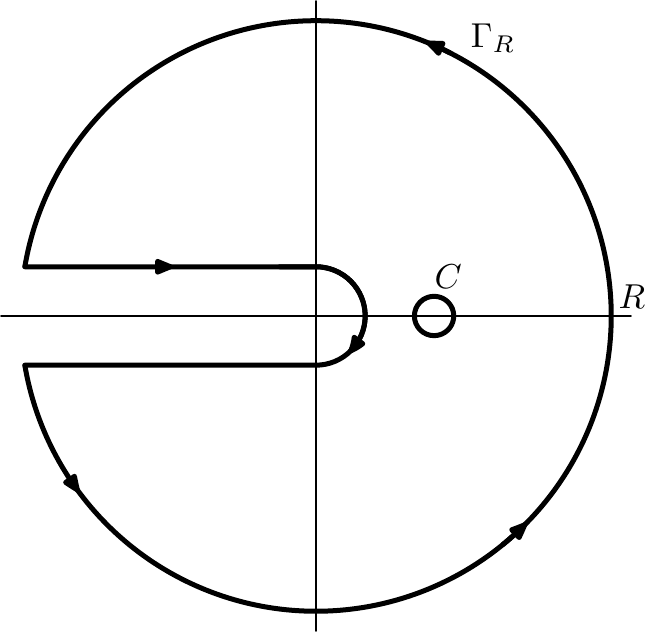}
%\caption{Yitang Zhang}
\end{figure}

\noindent where $C$ is a circle with center at $1+\frac{16}{3}u^2(1-u)^2)$ and radius $0<r<\frac12$ and 
$\Gamma_R$ is the contour in the figure formed with part of  the circle of radius $R>2$ and 
part of Hankel's path $H$.  When $R\to+\infty$ the integral along the portion of the circumference of radius $R$ tends to $0$. Therefore we obtain
\[\int_{H}\frac{dz}{\sqrt{z+\sqrt{z}}(1-z+\frac{16}{3}u^2(1-u)^2)}=\int_C\frac{dz}{\sqrt{z+\sqrt{z}}(z-1-\frac{16}{3}u^2(1-u)^2)}.\]
Notice that there is a change of sign because  $\Gamma_R$ contains a portion of $H$ 
in reverse sense.

The integral in $C$ is equal to the residue so that we get 
\[I=\int_0^1\frac{du}{\sqrt{1+\frac{16}{3}u^2(1-u)^2+\sqrt{1+\frac{16}{3}u^2(1-u)^2}}}\,du.\]
\end{proof}

\section{Simple transformations of the last integral.}

\begin{proposition}
\begin{equation}\label{E:7}
I=\int_0^1\frac{dx}{2\sqrt{1-x}}\frac{1}
{\sqrt{1+\frac{1}{3}x^2+\sqrt{1+\frac{1}{3}x^2}}}.
\end{equation}
\end{proposition}

\begin{proof}
In the integral \eqref{E:6} the integrand depends only of $x(1-x)$ so that the integral
is equal to
\[I=2\int_0^{1/2}\frac{dx}{\sqrt{1+\frac{16}{3}x^2(1-x)^2+\sqrt{1+\frac{16}{3}x^2(1-x)^2}}}\,dx.\]
Changing variables with  $x=\sin^2\theta$ 
\[I=2\int_0^{\pi/4}\frac{2\sin\theta\cos\theta\,d\theta}
{\sqrt{1+\frac{16}{3}\sin^4\theta\cos^4\theta+
\sqrt{1+\frac{16}{3}\sin^4\theta\cos^4\theta}}}.\]
This can be expressed in terms of the double angle
\[
I=\int_0^{\pi/4}
\frac{2\sin2\theta\,d\theta}
{\sqrt{1+\frac{1}{3}\sin^42\theta+\sqrt{1+\frac{1}{3}\sin^42\theta}}}.\]
The change of variables $x=1-\cos^22\theta$ transform this in \eqref{E:7}.
\end{proof}

\section{Integrand with only one square root.}

\begin{proposition}
\begin{equation}\label{E:8}
I=\frac{\sqrt{3}}{\sqrt{2-\sqrt{3}}}\int_0^{2-\sqrt{3}}
\Bigl(\frac{1-x+x^2}{x(1-x^2)(2-x)}\Bigr)^{1/2}\,\frac{dx}{2+\sqrt{3}-x}.
\end{equation}
\end{proposition}

\begin{proof}
The inner root in \eqref{E:7} can be rationalized. To this end we consider 
the hyperbola $x^2+3=y^2$. Using the rational point $(x,y)=(1,2)$ we get the 
rationalization
\[x=\frac{m^2-4m+1}{1-m^2},\quad y=\frac{2m-2m^2-2}{1-m^2}.\]
There are two intervals of $m$ that maps into $0<x<1$. They are 
$[0,2-\sqrt{3}]$ and $[2,2+\sqrt{3}]$, each one gives an adequate change of coordinates 
for our integral \eqref{E:7}. We will use the first one.
We have
\[\frac{1}{\sqrt{1-x}}=\sqrt{\frac{1-m^2}{2m(2-m)}},\quad 
1+x^2/3=\frac{4}{3}\Bigl(\frac{1+m(m-1)}{(1-m^2)}\Bigr)^2.\]
We want $\sqrt{1+x^2/3}$ to be positive. When  $0<m<2-\sqrt{3}$ we have $\frac{1+m(m-1)}{(1-m^2)}>0$
therefore we obtain 
\[1+x^2/3+\sqrt{1+x^2/3}=\frac{4}{3}\Bigl(\frac{1+m(m-1)}{(1-m^2)}\Bigr)^2+
\frac{2}{\sqrt{3}}\frac{1+m(m-1)}{(1-m^2)}.\]
This can be simplified to 
%\frac{2(1-m+m^2)((2-\sqrt{3})m^2-2m+2+\sqrt{3})}{3(1-m^2)^2}=
\[\frac{2(2-\sqrt{3})(1-m+m^2)(m-2-\sqrt{3})^2}{3(1-m^2)^2}.\]
Finally we have 
\[dx=-\frac{4(1-m+m^2)}{(1-m^2)^2}\,dm.\]
When $m$ runs through the interval $[0,2-\sqrt{3}]$ $x$ run from $1$ to $0$. Therefore
\begin{multline*}
I=\int_0^{2-\sqrt{3}}\frac12\sqrt{\frac{1-m^2}{2m(2-m)}}
\frac{\sqrt{3}(1-m^2)}{\sqrt{2(2-\sqrt{3})(1-m+m^2)}(2+\sqrt{3}-m)}\cdot\\
\cdot\frac{4(1-m+m^2)}{(1-m^2)^2}\,dm.
\end{multline*}
Simplifying we get \eqref{E:8}.
\end{proof}

\begin{remark}
Taking the interval $m\in(2,2+\sqrt{3})$ we obtain instead of \eqref{E:8}
\[I=\frac{\sqrt{3}}{\sqrt{2+\sqrt{3}}}\int_2^{2+\sqrt{3}}
\Bigl(\frac{1-x+x^2}{x(1-x^2)(2-x)}\Bigr)^{1/2}\,\frac{dx}{x-2+\sqrt{3}}.\]
\end{remark}

\section{Expression as a logarithmic integral.}

\begin{proposition}\label{P:8}
Let 
\begin{equation}\label{E:CB}
C_0=\frac{\sqrt{21+12\sqrt{3}}}{\sqrt{2}}\log\frac{3+2\sqrt{3}}{6},\quad 
B(t)=\frac{1+10t-\sqrt{1+32t+64t^2}}{8t}.
\end{equation}
Then we have 
\begin{equation}\label{E:9}
I=C_0+
\frac{2\sqrt{3}}{\sqrt{2-\sqrt{3}}}\int_{\frac18(2+\sqrt{3})}^\infty\log
\frac{2+\sqrt{3}}{\frac32+\sqrt{3}+\sqrt{B(t)}}\frac{dt}{2\sqrt{t}}.
\end{equation}
\end{proposition}

\begin{proof}
The zeros of the radical in \eqref{E:8}   are symmetric with respect two lines $\Re x=\frac12$ and $\Im x=0$. 
So we change variables by $x=\frac12-y$
\[I=\frac{2\sqrt{3}}{\sqrt{2-\sqrt{3}}}\int_{\sqrt{3}-\frac32}^{\frac12}
\Bigl(\frac{(3+4y^2)}{(1-4y^2)(9-4y^2)}\Bigr)^{1/2}\frac{dy}{\frac32+\sqrt{3}+y}.\]
Let $A(y)=\frac{(3+4y^2)}{(1-4y^2)(9-4y^2)}$. Our integral is 
\[I=\frac{2\sqrt{3}}{\sqrt{2-\sqrt{3}}}\int_{\sqrt{3}-\frac32}^{\frac12}\Bigl(\int_0^{A(y)}
\frac{dt}{2\sqrt{t}}\Bigr)\frac{dy}{\frac32+\sqrt{3}+y}.\]
Following Knuth we denote by $[t<A(y)]$ a function that is $1$ when $t<A(y)$ and $0$ 
in other case, using this we have 
\[I=\frac{2\sqrt{3}}{\sqrt{2-\sqrt{3}}}\int_{\sqrt{3}-\frac32}^{\frac12}\int_{0}^{+\infty} \frac{[t<A(y)]\,dt}{2\sqrt{t}}
\frac{dy}{\frac32+\sqrt{3}+y}.\]
We reverse now the order of integration.

For any fixed value of $t>0$ there are two values of $y^2$ with $A(y)=t$, 
\[y^2=\frac{1+10t-\sqrt{1+32t+64t^2}}{8t},\quad y^2=\frac{1+10t+\sqrt{1+32t+64t^2}}{8t}.\]
We are integrating for $y\in(\sqrt{3}-\frac32,\frac12)$, so we  are only 
interested in $0<y<\frac12$. In this range of $y$ the condition $t<A(y)$ is equivalent to 
$16y^4t-(40t-4)y^2+9t-3<0$. This implies that $y^2$ is contained between the two roots 
above. The second root 
gives $y^2>10/8$, so $y>1$ for this value.
It follows that for $0<y<\frac12$ 
\[t<A(y) \quad \text{is equivalent to}\quad y^2>B(t):=\frac{1+10t-\sqrt{1+32t+64t^2}}{8t}.\]
Therefore we have 
\[I=\frac{2\sqrt{3}}{\sqrt{2-\sqrt{3}}}\int_{\sqrt{3}-\frac32}^{\frac12}\int_{0}^{+\infty}\frac{[t<A(y)]}{2\sqrt{t}}
\frac{dy}{\frac32+\sqrt{3}+y}=
\int_0^\infty\frac{dt}{2\sqrt{t}}\int_J\frac{dy}{\frac32+\sqrt{3}+y}\]
where $J$ is the interval $(\sqrt{3}-\frac32,1/2)\cap(\sqrt{B(t)},\infty)$.

$\sqrt{B(t)}$ increases with $t$, and $\sqrt{B(t)}=\sqrt{3}-\frac32$ just for 
$t=\frac18(2+\sqrt{3})$. We have also $\sqrt{B(t)}<\frac12$ for $t>\frac18(2+\sqrt{3})$  so that 
\begin{multline*}
I=\frac{2\sqrt{3}}{\sqrt{2-\sqrt{3}}}\Bigl(
\int_0^{\frac18(2+\sqrt{3})}\frac{dt}{2\sqrt{t}}\int_{\sqrt{3}-\frac32}^{\frac12}\frac{dy}{\frac32+\sqrt{3}+y}\\+
\int_{\frac18(2+\sqrt{3})}^\infty\frac{dt}{2\sqrt{t}}\int_{\sqrt{B(t)}}^{\frac12}\frac{dy}{\frac32+\sqrt{3}+y}\Bigr).
\end{multline*}
After simplification we have
\[\int_{\sqrt{3}-\frac32}^{\frac12}\frac{dy}{\frac32+\sqrt{3}+y}=\log\frac{3+2\sqrt{3}}{6},\mskip15mu\frac{2\sqrt{3}}{\sqrt{2-\sqrt{3}}}
\int_0^{\frac18(2+\sqrt{3})}\frac{dt}{2\sqrt{t}}=\sqrt{\frac{21+12\sqrt{3}}2}.\]
\[I=\frac{\sqrt{21+12\sqrt{3}}}{\sqrt{2}}\log\frac{3+2\sqrt{3}}{6}+
\frac{2\sqrt{3}}{\sqrt{2-\sqrt{3}}}\int_{\frac18(2+\sqrt{3})}^\infty\log
\frac{2+\sqrt{3}}{\frac32+\sqrt{3}+\sqrt{B(t)}}\frac{dt}{2\sqrt{t}}.\]
\end{proof}

\section{Reduction to two elliptic integrals.}

\begin{proposition}\label{P:9}
We have $I=\frac{2\sqrt{3}}{\sqrt{2-\sqrt{3}}}J$, where 
\begin{multline}\label{E:10}
J=\frac12
\int_{4}^{4(3\sqrt{3}-4)}\Bigl(\frac{8-x}{x^2-16}\Bigr)^{1/2}\frac{3+2\sqrt{3}}
{4(4+3\sqrt{3})+x}\frac{dx}{\sqrt{5-x}}\\-
\frac12
\int_{4}^{4(3\sqrt{3}-4)}\Bigl(\frac{8-x}{x^2-16}\Bigr)^{1/2}\frac{dx}
{4(4+3\sqrt{3})+x}.
\end{multline}
\end{proposition}

\begin{proof}
By Proposition \ref{P:8} we have $I=C_0+\frac{2\sqrt{3}}{\sqrt{2-\sqrt{3}}}U$ where 
\[U=\int_{\frac18(2+\sqrt{3})}^\infty\log
\frac{2+\sqrt{3}}{\frac32+\sqrt{3}+\sqrt{B(t)}}\frac{dt}{2\sqrt{t}}.\]
$\sqrt{B(t)}$ contains the inner root $\sqrt{1+32t+64t^2}$. We rationalize this 
square root by the change
\[t=\frac{m-8}{16-m^2},\qquad  1+32t+64t^2=\Bigl(\frac{16-16m+m^2}{16-m^2}\Bigr)^2.\]
The value of $\frac{m-8}{16-m^2}$ run through the interval $(\frac18(2+\sqrt{3}),\infty)$
for two intervals of $m$. First for $m\in(-4\sqrt{3},-4)$, and second for $m\in(4,4(3\sqrt{3}-4))$ (this time in reverse order).
The square root $\sqrt{B(t)}$ is different in the two intervals. We obtain
\begin{align*}
B\Bigl(\frac{m-8}{16-m^2}\Bigr)&=\frac{5-m}{4},\quad 4<m<4(3\sqrt{3}-4);\\
B\Bigl(\frac{m-8}{16-m^2}\Bigr)&=\frac{3(8+m)}{4(8-m)},\quad-4\sqrt{3}<m<-4.
\end{align*}
Therefore
\[U=-\int_{4}^{4(3\sqrt{3}-4)}\log
\frac{2+\sqrt{3}}{\frac32+\sqrt{3}+\frac12\sqrt{5-m}}\,d\Bigl(\frac{8-m}{m^2-16}\Bigr)^{1/2}.\]
This can be written
\[U=\int_{4}^{4(3\sqrt{3}-4)}\log\Bigl(1+\frac{2-\sqrt{3}}{2}(\sqrt{5-m}-1)\Bigr)
\,d\Bigl(\frac{8-m}{m^2-16}\Bigr)^{1/2}.\]
Integrating by parts
\begin{multline*}
U=\frac{1+\sqrt{3}}{4}\log(4\sqrt{3}-6)\\+\int_{4}^{4(3\sqrt{3}-4)}\Bigl(\frac{8-m}{m^2-16}\Bigr)^{1/2}\frac{1}{1+\frac{2-\sqrt{3}}{2}(\sqrt{5-m}-1)} \frac{2-\sqrt{3}}{2}
\frac{1}{2\sqrt{5-m}}\,dm,
\end{multline*}
simplifying and changing the name of the variable
\begin{multline*}
U=\frac{1+\sqrt{3}}{4}\log(4\sqrt{3}-6)\\+\frac12
\int_{4}^{4(3\sqrt{3}-4)}\Bigl(\frac{8-x}{x^2-16}\Bigr)^{1/2}\frac{1}{3+2\sqrt{3}+\sqrt{5-x}} \frac{1}{\sqrt{5-x}}\,dx.
\end{multline*}
This integral is the sum of other two
\begin{align*}
U&=\frac{1+\sqrt{3}}{4}\log(4\sqrt{3}-6)\\
&\mskip60mu+\frac12
\int_{4}^{4(3\sqrt{3}-4)}\Bigl(\frac{8-x}{x^2-16}\Bigr)^{1/2}\frac{3+2\sqrt{3}}
{4(4+3\sqrt{3})+x}\frac{1}{\sqrt{5-x}}\,dx\\
&\mskip120mu-
\frac12
\int_{4}^{4(3\sqrt{3}-4)}\Bigl(\frac{8-x}{x^2-16}\Bigr)^{1/2}\frac{1}
{4(4+3\sqrt{3})+x}\,dx,\\
&:=\frac{1+\sqrt{3}}{4}\log(4\sqrt{3}-6)+J,
\end{align*}
(notice that  $4(4+3\sqrt{3})=(3+2\sqrt{3})^2-5$).

We have proved that $I=C+\frac{2\sqrt{3}}{\sqrt{2-\sqrt{3}}}J$ where 
\[C=C_0+ \frac{2\sqrt{3}}{\sqrt{2-\sqrt{3}}}\frac{1+\sqrt{3}}{4}\log(4\sqrt{3}-6).\]
But this constant is equal to $0$. 
To prove it notice first that 
\[\log\frac{3+2\sqrt{3}}{6}+\log(4\sqrt{3}-6)=\log 1=0.\]
Then we only have to check that 
\[\frac{\sqrt{21+12\sqrt{3}}}{\sqrt{2}}=\frac{2\sqrt{3}}{\sqrt{2-\sqrt{3}}}\frac{1+\sqrt{3}}{4}.\]
This is equivalent to 
$\frac{4\sqrt{6+3\sqrt{3}}}
{2\sqrt{2}(3+\sqrt{3})}=1$, that is checked by squaring.
\end{proof}

\section{Reduction to normal form (modulus $k=2-\sqrt{3}$).}

The integrals in Proposition \ref{P:9} are elliptic integrals. That is integrals 
$\int R(x,y)\,dx$ where $R$ is rational and $y=\sqrt{p(x)}$ where $p(x)$ is a polynomial
of degree $3$ or $4$ without multiple roots. It is known that these integrals can be 
reduced to three canonical forms. 

Our  two elliptic integrals are in fact corresponding to the same modulus $k=2-\sqrt{3}$.

\begin{proposition}\label{P:firstform}
The integral $I=aJ_1+bJ_2$ with 
\begin{equation}
\begin{aligned}
J_1&=\int_{\frac{1+\sqrt{3}}{2}}^{2+\sqrt{3}}\frac{x+1}{x+1+\sqrt{3}}\frac{dx}{\sqrt{(x^2-1)(1-k^2x^2)}},
\\
J_2&=\int_1^{\frac{1+\sqrt{3}}{2}}
\frac{x-2-\sqrt{3}}{x+1+\sqrt{3}}\frac{dx}{\sqrt{(x^2-1)(1-k^2x^2)}},
\end{aligned}
\end{equation}
where $k=2-\sqrt{3}$, $a$ and
$b$ the algebraic numbers
\begin{equation}
a =\frac{\sqrt{3}}{2\sqrt{2}}, \quad b=\frac{2\sqrt{3}-3}{2\sqrt{2}}.
\end{equation}
\end{proposition}

\begin{proof}
Consider the first integral in \eqref{E:10}
\[H_1=\frac12
\int_{4}^{4(-4+3\sqrt{3})}\Bigl(\frac{8-x}{x^2-16}\Bigr)^{1/2}\frac{3+2\sqrt{3}}
{4(4+3\sqrt{3})+x}\frac{1}{\sqrt{5-x}}\,dx.\]
We reduce the radical to the usual form. Following the general theory, we change variables by means of the bilinear transformation
\[x=L(t):=\frac{4}{13}(14+3\sqrt{3})\frac{t-\frac{14-3\sqrt{3}}{13}}
{t-\frac{14+3\sqrt{3}}{13}}\]
$L$ send the points $-1/k$, $-1$, $1$ and $1/k$ into $5$, $4$, $-4$, $8$   respectively, where
$k=2-\sqrt{3}$ and 
\[(x^2-16)(8-x)(5-x)\to -(2-\sqrt{3})\frac{2^7 3^6(1-t^2)(1-k^2t^2)}{(3-4\sqrt{3}+(5\sqrt{3}-6)t)^4}.\]
The interval of integration $4<x<4(-4+3\sqrt{3})$ corresponds to $-1>t>-1-\sqrt{3}$.
\[\frac{(3+2\sqrt{3})(8-x)}
{4(4+3\sqrt{3})+x}=\frac{3-\sqrt{3}}{4}\frac{t-2-\sqrt{3}}{t-\frac12(1+\sqrt{3})}.\]
\[dx=-\frac{24(9+14\sqrt{3})}{(13t-(14+3\sqrt{3}))^2}\,dt.\]
After simplification we get 
\[H_1=\frac{11\sqrt{3}-5}{8\sqrt{2}\sqrt{698+391\sqrt{3}}}\int_{-1-\sqrt{3}}^{-1}\frac{1}
{\sqrt{(t^2-1)(1-k^2t^2)}}
\frac{(2+\sqrt{3}-t)}{(\frac12(1+\sqrt{3})-t)}\,dt.\]
Therefore  the coefficient of $\frac{2\sqrt{3}}{\sqrt{2-\sqrt{3}}}H_1$ is equal to 
\[\frac{2\sqrt{3}}{\sqrt{2-\sqrt{3}}}\cdot\frac{11\sqrt{3}-5}{8\sqrt{2}\sqrt{698+391\sqrt{3}}}=\frac{3-\sqrt{3}}{4\sqrt{2}}.\]
That is easily checked after we notice that $\sqrt{2-\sqrt{3}}\sqrt{698+391\sqrt{3}}=14+3\sqrt{3}$.

Changing variables $t= -\frac{1}{kx}$ we obtain 
\[\frac{2\sqrt{3}}{\sqrt{2-\sqrt{3}}}H_1=\frac{\sqrt{3}}{2\sqrt{2}}
\int_{\frac{1+\sqrt{3}}{2}}^{2+\sqrt{3}}\frac{x+1}{x+1+\sqrt{3}}\frac{dx}{\sqrt{(x^2-1)(1-k^2x^2)}}.\]

The second integral in \eqref{E:10} is 
\[H_2=\frac12
\int_{4}^{4(3\sqrt{3}-4)}\Bigl(\frac{8-x}{x^2-16}\Bigr)^{1/2}\frac{1}
{4(4+3\sqrt{3})+x}\,dx.\]
The change of variables
\[x=L(t):=4(2+\sqrt{3})\frac{t-2+\sqrt{3}}{t-2-\sqrt{3}},\]
sends $-1/k$, $-1$, $1$, $1/k$ into $8$, $4$, $-4$ and $\infty$. We have also
for $4<x<4(3\sqrt{3}-4)$
\[\sqrt{(x^2-16)(8-x)}=\sqrt{(t^2-1)(1-k^2t^2)}\frac{8(9+5\sqrt{3})}{(t-2-\sqrt{3})^2},\]
\[\frac{8-x}{4(4+3\sqrt{3})+x}\frac{(t-2-\sqrt{3})^2}{8(9+5\sqrt{3})}L'(t)=
\frac{3\sqrt{3}-5}{4}\frac{t+2+\sqrt{3}}{t-1-\sqrt{3}},\]
\[L(-1)=4,\qquad L(-(1+\sqrt{3})/2)=4(3\sqrt{3}-4).\] 
Therefore after simplifications we get 
\[H_2=-\frac{3\sqrt{3}-5}{8}\int_{-\frac{1+\sqrt{3}}{2}}^{-1}
\frac{t+2+\sqrt{3}}{t-1-\sqrt{3}}\frac{dt}{\sqrt{-(1-t^2)(1-k^2t^2)}}.\]
The coefficient here by the factor of $J$ in Proposition \ref{P:9} is
\[b=\frac{2\sqrt{3}}{\sqrt{2-\sqrt{3}}}\cdot\frac{3\sqrt{3}-5}{8}=\frac{2\sqrt{3}-3}{2\sqrt{2}}.\]
Changing  variables  $t=-x$ yields
\[\frac{2\sqrt{3}}{\sqrt{2-\sqrt{3}}}H_2=
-\frac{2\sqrt{3}-3}{2\sqrt{2}}\int_1^{\frac{1+\sqrt{3}}{2}}
\frac{x-2-\sqrt{3}}{x+1+\sqrt{3}}\frac{dx}{\sqrt{(x^2-1)(1-k^2x^2)}},\]
and we get our result $I=aJ_1+bJ_2$.
\end{proof}

\begin{theorem}\label{T:dos}
Let $k=2-\sqrt{3}$, $a=(1+\sqrt{3})/2=\frac{1}{1-k}$ and $\Delta=(x^2-1)(1-k^2x^2)$ then 
\begin{equation}\label{E:threeintegrals}
2\sqrt{2}\;I=\sqrt{3}\int_1^{1/k}\frac{dx}{\sqrt{\Delta}}+
(\sqrt{3}-3)\int_1^a \frac{dx}{\sqrt{\Delta}}-3
\int_1^{1/k}\frac{1}{x+1+\sqrt{3}}\frac{dx}{\sqrt{\Delta}}.
\end{equation}
\end{theorem}

\begin{proof}
By Proposition \ref{P:firstform} we have
\begin{align*}
2\sqrt{2}I&=\int_a^{1/k}\frac{\sqrt{3}(x+1)}{x+1+\sqrt{3}}\frac{dx}{\sqrt{\Delta}}+
\int_1^a \frac{(2\sqrt{3}-3)x-\sqrt{3}}{x+1+\sqrt{3}}\frac{dx}{\sqrt{\Delta}}\\
&=\int_1^{1/k}\frac{\sqrt{3}(x+1)}{x+1+\sqrt{3}}\frac{dx}{\sqrt{\Delta}}+
\int_1^a \frac{(\sqrt{3}-3)x-2\sqrt{3}}{x+1+\sqrt{3}}\frac{dx}{\sqrt{\Delta}}\\
&=\int_1^{1/k}\frac{\sqrt{3}(x+1)}{x+1+\sqrt{3}}\frac{dx}{\sqrt{\Delta}}+
(\sqrt{3}-3)\int_1^a \frac{dx}{\sqrt{\Delta}}\\
&=\sqrt{3}\int_1^{1/k}\frac{dx}{\sqrt{\Delta}}-
\int_1^{1/k}\frac{3}{x+1+\sqrt{3}}\frac{dx}{\sqrt{\Delta}}+
(\sqrt{3}-3)\int_1^a \frac{dx}{\sqrt{\Delta}}.\\
\end{align*}
\end{proof}

\section{Reduction to normal form with modulus $1/\sqrt{3}$}

We will use the notation of Gradshteyn and Ryzhik for the three normal forms of elliptic integrals.  They are 
\begin{multline*}
\Pi(\varphi,n^2,k)=\int_0^\varphi\frac{da}{(1-n^2\sin^2a)\sqrt{1-k^2\sin^2a}}\\=
\int_0^{\sin\varphi}\frac{dx}{(1-n^2 x^2)\sqrt{(1-x^2)(1-k^2x^2)}},\qquad 
(-\infty<n^2<\infty).
\end{multline*}
\[ F(\varphi,k)=\int_0^\varphi\frac{d\alpha}{\sqrt{1-k^2\sin^2\alpha}}=
\int_0^{\sin\varphi}\frac{dx}{\sqrt{(1-x^2)(1-k^2x^2)}}.\]
\[ E(\varphi,k)=\int_0^\varphi\sqrt{1-k^2\sin^2\alpha}\,d\alpha=
\int_0^{\sin\varphi}\frac{\sqrt{1-k^2 x^2}}{\sqrt{1-x^2}}\,dx.\]
Here $k$ is the modulus. 

The integrals appearing in Theorem \ref{T:dos} are real but the integrand contains
$\sqrt{\Delta(x)}$ with $\Delta(x)=(x^2-1)(1-k^2x^2)$ by a change of variables we may 
consider real integrals with the standard form of the radicand $(1-x^2)(1-\ell^2x^2)$.
In our case $\ell^2$ will be $1/3$.  All this can be done by standard transformations. In particular
I find very useful Byrd and Friedman \cite{BF} Handbook of Elliptic Integrals.

The first integral  in \eqref{E:threeintegrals} is well known in \cite{WW}*{p.~501} we find
\[\int_1^{1/k}\frac{dx}{\sqrt{\Delta}}=K'=K(k'),\]
where the complementary modulus $k'$ is defined by $k^2+k'^2=1$. 
For the other two integrals we use the tables of Byrd and Friedman.

\begin{proposition}\label{P:valuesInt}
Let $k=2-\sqrt{3}$, we have 
\begin{equation}
\int_1^a \frac{dx}{\sqrt{\Delta}}=\frac{3+\sqrt{3}}{3}\int_0^{\sqrt{k}}\frac{dx}{\sqrt{(1-x^2)(1-\frac13x^2)}}.
\end{equation}
\begin{multline}
\int_1^{1/k}\frac{1}{x+1+\sqrt{3}}\frac{dx}{\sqrt{\Delta}}\\=
\frac{2(1+\sqrt{3})}{3}\int_0^1\frac{dx}{(1-\frac{k}{3}x^2)\sqrt{(1-x^2)(1-\frac13 x^2)}}-
\frac{3+\sqrt{3}}{3}K(1/\sqrt{3}).
\end{multline}
\end{proposition}

\begin{proof}
According to \cite{BF}*{\textbf{256.00} p.~120} we have
\[
\int_1^a \frac{dx}{\sqrt{\Delta}}=\frac{1}{k}\int_1^a\frac{dx}
{\sqrt{(\frac1k-x)(x-1)(x+1)(x+\frac1k)}}
=\frac{g}{k} F(\varphi,k_1),
\]
where
\[k_1^2=\frac{(\frac1k-1)(\frac1k-1)}{(\frac1k+1)(\frac1k+1)}\quad\therefore\quad k_1=\frac{1}{\sqrt{3}}\]
\[\varphi=\arcsin \sqrt{\frac{(\frac1k+1)(a-1)}{(\frac1k-1)(a+1)}}
\quad\therefore\quad \varphi=\arcsin\sqrt{k}.\]
\[g=\frac{2}{\sqrt{(\frac1k+1)(1+\frac1k)}}\quad\therefore\quad g=\frac{3-\sqrt{3}}{3},
\quad \frac{g}{k}=\frac{3+\sqrt{3}}{3}.\]
In the same way \cite{BF}*{\textbf{256.39} p.~124} with $m=1$ gives
\[\int_1^{1/k}\frac{1}{x+1+\sqrt{3}}\frac{dx}{\sqrt{\Delta}}=\frac{g}{(2+\sqrt{3})k}
\int_0^{u_1}\frac{1-\alpha^2\sn^2 u}{1-\alpha_3^2\sn^2u}\,du\]
where the modulus of the function $\sn u$ is $k_1=3^{-1/2}$,  $g$ is given above,   $\sn u_1=\sin\varphi$,  $\alpha^2=\frac{\frac1k-1}{\frac1k+1}=\frac{1}{\sqrt{3}}$,
\[\varphi=\arcsin\sqrt{\frac{(\frac1k+1)(\frac1k-1)}{(\frac1k-1)(\frac1k+1)}}=\arcsin1=\frac{\pi}{2},\quad \sn u_1=1, \quad\therefore\quad u_1=K(k_1).\]
and 
\[\alpha_3^2= \frac{(-1-\sqrt{3}+1)(\frac1k-1)}{(-1-\sqrt{3}-1)(\frac1k+1)},
\quad\therefore\quad \alpha_3^2=2-\sqrt{3}.\] 
The primitive is given in \cite{BF}*{\textbf{340.01} p.~205}. It follows that 
\[\int_1^{1/k}\frac{1}{x+1+\sqrt{3}}\frac{dx}{\sqrt{\Delta}}=
\frac{3-\sqrt{3}}{3}\frac{1}{2-\sqrt{3}}\Bigl[(2-\sqrt{3}-\tfrac{1}{\sqrt{3}})
\Pi(\tfrac{\pi}{2},2-\sqrt{3},k_1)+\tfrac{1}{\sqrt{3}}u_1\Bigr].\]
So that 
\[\int_1^{1/k}\frac{1}{x+1+\sqrt{3}}\frac{dx}{\sqrt{\Delta}}=
\frac{1+\sqrt{3}}{3}K(1/\sqrt{3})-\frac{2(\sqrt{3}-1)}{3}\Pi(\tfrac{\pi}{2}, k, 1/\sqrt{3}).\]
\end{proof}

We state now our main result.
\begin{theorem}
We have with $k=2-\sqrt{3}$ 
\begin{equation}\label{E:final}
\sqrt{2}\;I=(\sqrt{3}-1)
\Pi(\tfrac{\pi}{2}, k, 3^{-1/2})-F(\alpha,3^{-1/2}),
\end{equation}
where $\alpha=\arcsin\sqrt{k}$. 
\end{theorem}

\begin{proof}
Substituting in \eqref{E:threeintegrals} the values given in Proposition \ref{P:valuesInt}
yields 
\[
2\sqrt{2}I=\sqrt{3}K(k')-2F(\alpha,3^{-1/2})-(1+\sqrt{3})K(3^{-1/2})+2(\sqrt{3}-1)
\Pi(\tfrac{\pi}{2}, k, 3^{-1/2}),
\]

But by the descending Landen transformation \href{http://dlmf.nist.gov/19.8.E12}{\cite{DLMF}*{\textbf{19.8.12}}}
\[K(\sqrt{1-k^2})=\frac{2}{1+k}K\Bigl(\frac{1-k}{1+k}\Bigr),\]
with $k=2-\sqrt{3}$ yields 
\[\sqrt{3}K(k')=(1+\sqrt{3})K(3^{-1/2}).\]
Therefore two terms cancel in our equation and \eqref{E:final} follows.
\end{proof}


\begin{thebibliography}{99}

\bibitem{AmdMoll}
T. Amdeberhan and V.~Moll, 
\newblock {\em The integrals in Gradshteyn and Ryzhik. Part 14: An elementary evaluation of entry 3.411.5},
\newblock Sci. Ser. A Math. Sci. (N.S.) \textbf{19} (2010), 97--103. 

\bibitem{Borwein}
D.~H. Bailey, J.~M.~Borwein, N.~J.~Calkin, R.~Girgensohn, D.~R.~Luke, V.~Moll, 
\newblock {\em Experimental mathematics in action},  A K Peters, Ltd., Wellesley, MA, 2007.  

\bibitem{BM}
G.~Boros,  V.~Moll and S.~Riley, 
\newblock {\em An elementary evaluation of a quartic integral},
\newblock Sci. Ser. A Math. Sci. (N.S.) \textbf{11} (2005), 1--12. 


\bibitem{BF}
\newblock P.~F.~Byrd and M.~D.~Friedman,
\newblock Handbook of Elliptic Integrals for Engineers and Scientist,
\newblock Second Ed., Revised, 
\newblock Springer-Verlag, Berlin, 1971.



\bibitem{gr1}
I.~S. Gradshteyn and I.~M. Ryzhik.
\newblock {\em Table of {I}ntegrals, {S}eries, and {P}roducts}.
\newblock Edited by A. Jeffrey and D. Zwillinger. Academic Press, New York, 6th
edition, 2000.

\bibitem{gr2}
I.~S. Gradshteyn and I.~M. Ryzhik.
\newblock {\em Table of {I}ntegrals, {S}eries, and {P}roducts}.
\newblock Edited by A. Jeffrey and D. Zwillinger. Academic Press, New York, 7th
edition, 2007.



%\bibitem{moll}
%V.~Moll.
%\newblock The Evaluation of Integrals: A Personal Story.
%\newblock {\em Notices Amer. Math. Soc.}, 3:311--317, 2002.

\bibitem{moll2}
V.~Moll.
\newblock Seized Opportunities.
\newblock {\em Notices Amer. Math. Soc.}, \textbf{57}, 4:476--484, 2010.


\bibitem{mollB1}
V.~Moll.
\newblock {\em Special Integrals of Gradshteyn and Ryzhik: the Proofs - Volume I.}
\newblock CRC Press, Boca Raton, FA, 2015.

\bibitem{mollB2}
V.~Moll.
\newblock {\em Special Integrals of Gradshteyn and Ryzhik: the Proofs - Volume II.}
\newblock CRC Press, Boca Raton, FA, 2016.

\bibitem{DLMF}
\newblock {\em NIST Digital Library of Mathematical Functions},
\newblock \href{http://dlmf.nist.gov/}{http://dlmf.nist.gov/}
\newblock F.~W.~J. Olver, A.~B. {Olde Daalhuis}, D.~W. Lozier, B.~I. Schneider,
                R.~F. Boisvert, C.~W. Clark, B.~R. Miller and B.~V. Saunders, eds.


\bibitem{WW}
E.~T.~Whittaker and G.~N.~Watson, 
\newblock A Course in Modern Analysis, Fourth Ed., 
\newblock Cambridge University Press, New York, 1965.


\end{thebibliography}
\end{document}